\documentclass[oribibl,envcountsame]{llncs}

\usepackage[ruled,vlined]{algorithm2e}
\usepackage{amsfonts}
\usepackage{amssymb}
\usepackage{mathtools}
\usepackage[all,cmtip]{xy}

\newcommand{\Rea}{\ensuremath{\mathbb{R}}}

\newtheorem{claimN}{Claim}

\renewcommand{\DeclareMathOperator}[1]{\newcommand{#1}}

\DeclareMathOperator{\Ord}{\mathrm{On}}

\DeclareMathOperator{\SBSS}{\ensuremath{\mathsf{SBSSM}}}
\DeclareMathOperator{\BSS}{\ensuremath{\mathsf{BSSM}}}

\DeclareMathOperator{\OTM}{\ensuremath{\mathsf{OTM}}}
\DeclareMathOperator{\ITTM}{\ensuremath{\mathsf{ITTM}}}

\DeclareMathOperator{\ITRM}{\ensuremath{\mathsf{ITRM}}}

\DeclareMathOperator{\ITBM}{\ensuremath{\mathsf{ITBM}}}
\DeclareMathOperator{\SITBM}{\ensuremath{\mathsf{SITBM}}}
\DeclareMathOperator{\BSITBM}{\ensuremath{\mathsf{BSITBM}}}
\DeclareMathOperator{\WITBM}{\ensuremath{\mathsf{WITBM}}}
\DeclareMathOperator{\SSITBM}{\ensuremath{\mathsf{SSITBM}}}

\DeclareMathOperator{\LimP}{\ensuremath{\mathrm{LimP}}}
\DeclareMathOperator{\SLimP}{\ensuremath{\mathrm{sLimP}}}

\usepackage{color}
\definecolor{lightgray}{gray}{0.75}

\title{Resetting Infinite Time Blum-Shub-Smale-Machines}
\author{Merlin Carl\inst{1} \and Lorenzo Galeotti\inst{2}}
\institute{Europa-Universit\"at Flensburg, 24943 Flensburg, Germany \and Amsterdam University 
College, Postbus 94160, 1090 GD Amsterdam, The Netherlands}

\begin{document}
\maketitle
\begin{abstract}
In this paper, we study strengthenings of Infinite Times Blum-Shub-Smale-Machines ($\ITBM$s) that were proposed by Seyfferth in \cite{Seyfferth} and Welch in \cite{Welch2014} obtained by modifying the behaviour of the machines at limit stages. In particular, we study Strong Infinite Times Blum-Shub-Smale-Machines ($\SITBM$s), a variation of $\ITBM$s where $\lim$ is substituted by $\liminf$ in computing the content of registers at limit steps. We will provide lower bounds to the computational strength of such machines. Then, we will study the computational strength of restrictions of $\SITBM$s whose computations have low complexity. We will provide an upper bound to the computational strength of these machines, in doing so we will strenghten a result in \cite{Welch2014} and we will give a partial answer to a question posed by Welch in \cite{Welch2014}.   
\end{abstract}
\setcounter{footnote}{0}
\section{Introduction}
In \cite{BSS} Blum, Shub and Smale introduced register machines which compute over real numbers called Blum-Shub-Smale-Machines ($\BSS$s). A $\BSS$ is a register machine whose registers contain real numbers and that at each step of the computation can either check the content of the registers and perform a jump based on the result, or apply a rational function to the registers. While being quite powerful, $\BSS$s are still bound to work for a finite amount of time. This limitation is in contrast with the fact that every real number is usually thought as to encode an infinite amount of information. It is therefore natural to ask if a transfinite version of these machines is possible. An answer to this question was first given by Koepke and Seyfferth who in \cite{Seyfferth} and \cite{Koepke12} introduced the notion of Infinite Time Blum, Shub and Smale machine ($\ITBM$). These machines execute classical $\BSS$-programs\footnote{With the limitation that every rational polynomial appearing in the program has rational coefficients.} but are capable of running for a transfinite amount of time. More precisely, the behaviour of $\ITBM$s at successor stages is completely analogous to that of a normal $\BSS$. At limit stage an $\ITBM$ computes the content of each register using the limit operation on $\mathbb{R}$, and updates the program counter using inferior limits. Infinite Time Blum-Shub-Smale-Machines were further studied in \cite{Koepke2017}.  

As mentioned in \cite{Galeotti2019}, the approach taken in extending $\BSS$s to $\ITBM$s is analogous to  that used by Hamkins and Lewis in \cite{Hamkins} and by Koepke and Miller in \cite{Koepke2017} to introduce Infinite Time Turing machines ($\ITTM$s) and Infinite Time Register machines ($\ITRM$s), respectively. A different approach to the generalisation of $\BSS$s analogous to that used by Koepke in \cite{Koepke09} to introduce Ordinal Turing Machines ($\OTM$) was taken by the second author in \cite{Galeotti2019, GaleottiPhD} where he introduced the notion of Surreal Blum, Shub and Smale machine ($\SBSS$). A complete account of all these models of computation can be found in \cite{CarlBook}. 

As shown in \cite{Welch2016,Koepke2017}, because of the way in which they compute the content of registers at limit stages, $\ITBM$s have a limited computational power, which is way below that of other notions of transfinite computability such as $\ITTM$s, $\OTM$s, $\SBSS$s, $\ITRM$s. In this paper we consider strengthenings of the limit rule for $\ITBM$s. We will first consider four natural modifications of the behaviour of $\ITBM$s at limit stages. For each such modification we show that it can either be simulated by $\ITBM$s or by machines that we will call Strong Infinite Time Register machines ($\SITBM$s). In the first part of the paper we will provide a lower bound to the computational power of $\SITBM$s. Then, we will restict our attention to programs whose $\SITBM$-computation is of low complexity. We will show that $\Pi_3$-reflection is an upper bound to the computational strength of these low complexity machines, in doing so we will strengthen a result mentioned without proof by Welch in \cite{Welch2014} and we will give a partial negative answer to the question asked by Welch of whether $\Pi_3$-reflection is an optimal upper bound to the halting times of $\BSS$-programs whose $\SITBM$-computation uses only rational numbers.

Many of the arguments in this paper are inspired by those in \cite{CarlComp} and \cite{CarlZOOII}.
\section{$\ITBM$ Limit Rules}

As we mentioned in the introduction, the value of registers of an $\ITBM$ at limit stages is computed using Cauchy limits, i.e., the content of each register at limit steps is the limit of the sequence obtained by considering the content of the register at previous stages of the computation. The machine is assumed to diverge if for at least one of the registers this limit does not exist. In this section we will consider modified versions of the limit behaviour of $\ITBM$s.

Let $f:\alpha\rightarrow \mathbb{R}$ be an $\alpha$-sequence on $\mathbb{R}$.  We call $\ell\in \mathbb{R}$ a \emph{finite limit} point of $f$ if for all $\beta<\alpha$ and for all positive $\varepsilon\in \mathbb{R}$ there is $\beta<\gamma$ such that $|f(\gamma)-\ell|<\varepsilon$. If for all $x\in \mathbb{R}$ and for all $\beta<\alpha$ there is $\beta<\gamma<\alpha$ such that $x<f(\gamma)$ then we will say that $+\infty$ is an \emph{infinite limit} point of $f$. Similarly for $-\infty$. If for all $x\in \mathbb{R}$ there is $\beta<\alpha$ such that for all $\beta<\gamma<\alpha$ we have $x<f(\gamma)$ then we will say that $+\infty$ is an \emph{infinite strong limit} point of $f$. Similarly, if for all $x\in \mathbb{R}$ there is $\beta<\alpha$ such that for all $\beta<\gamma<\alpha$ we have $x>f(\gamma)$ then we will say that $-\infty$ is an \emph{infinite strong limit} point of $f$.

We will denote by $\LimP(f)\subset \mathbb{R}\cup\{-\infty,+\infty\}$ the set of finite and infinite limit points of $f$, and by $\SLimP(f)\subset \mathbb{R}\cup\{-\infty,+\infty\}$ the set of finite limit points and infinite strong limit points of $f$. Note that, if $f$ has an infinite strong limit point, then $\LimP(f)$ is a singleton containing either $-\infty$ or $+\infty$.

\begin{remark}\label{Rem1}
If $f$ is Cauchy with limit $\ell$, then $\SLimP(f)=\LimP(f)=\{\ell\}$. The converse is not true. Indeed, fix $\rho:\omega\times \omega \rightarrow \omega$ to be any computable bijection. Let $f_n(m)=\ell+\frac{n}{m}$ for $n>0$. Each $f_n$ is a countable sequence which converges to $\ell$. Then, setting $s(\rho( n,m))=f_n(m)$  we see that the sequence $s$ is such that $\SLimP(s)=\{\ell\}$ but is not Cauchy.  
\end{remark}

\begin{remark}\label{Rem2}
Note that if $\alpha$ is a limit ordinal, $f:\alpha\rightarrow \Rea$ is an $\alpha$-sequence, $\gamma<\alpha$ is an ordinal, and $g$ is the sequence $g(\beta)=f(\gamma+\beta)$, then $\LimP(f)=\LimP(g)$, and $\SLimP(f)=\SLimP(g)$. 
\end{remark}

Let $\lambda$ be a limit ordinal and $R_i$ be a register. Assume that for each $\alpha<\lambda$ the register $R_i$ had content $R_i(\alpha)$ in the $\alpha$th step of the computation. We  consider $\ITBM$s whose limit rule is modified as follows:
\begin{enumerate}
\item \emph{Weak $\ITBM$} ($\WITBM$): 
$$R_i(\lambda)=\begin{cases} \ell &\text{if $R^{\lambda}_i$ is Cauchy with limit $\ell$; }\\
0 &\text{if $\min(\SLimP(R^{\lambda}_i))$ is an infinite strong limit.}
\end{cases}$$
\item \emph{Strong $\ITBM$} ($\SITBM$): $$R_i(\lambda)=\begin{cases}
\min(\LimP(R^{\lambda}_i)) &\text{if $\min(\SLimP(R^{\lambda}_i))\in \mathbb{R}$; }\\
0 &\text{if $\min(\SLimP(R^{\lambda}_i))$ is an infinite strong limit.}
\end{cases}$$
\item \emph{Bounded-strong $\ITBM$} ($\BSITBM$): $$R_i(\lambda)= \min(\SLimP(R^{\lambda}_i)) \text{ if $\min(\SLimP(R^{\lambda}_i))\in \mathbb{R}$.}$$
\item  \emph{Super-strong $\ITBM$} ($\SSITBM$): $$R_i(\lambda)=\begin{cases} \min(\LimP(R^{\lambda}_i)) & \text{if $\LimP(R^{\lambda}_i)\in \mathbb{R}$;} \\
0 &\text{if $\min(\LimP(R^{\lambda}_i))$ is infinite.}
\end{cases}
$$
\end{enumerate}
If $\LimP(R_i)=\emptyset$, $\SLimP(R_i)=\emptyset$, or in general if $R_i(\lambda)$ is not defined, we will assume that the machine crashes, i.e., the computation is considered divergent. In each case the rest of the machine is left unchanged. We will also consider what happens when the register contents of an $\SITBM$s are required to belong to a restricted set of real numbers; for $X\subseteq \mathbb{R}$, an $X$-$\SITBM$ works like an $\SITBM$, but the computation is undefined when a state arises in which some register content is not contained in $X$.  

As for $\ITBM$s all the machines we consider run $\BSS$-programs and work on real numbers. Therefore, given $\Gamma\in\{\SITBM,\BSITBM,\SSITBM\}$, the definitions of \emph{$\Gamma$-computation}, \emph{$\Gamma$-computable function}, \emph{$\Gamma$-computable set}, and \emph{$\Gamma$-clockable ordinal} are exactly the same as the correspondent notions for $\ITBM$s, see, e.g., \cite[Definition 1, Definition 2, Definition 3]{Koepke12}. As noted in \cite[Algorithm 4]{Koepke12} $\ITBM$s are capable of computing the binary representation of a real number and perform local changes to this infinite binary sequence. The same algorithms work for $\SITBM$, and therefore for $\BSITBM$, and $\SSITBM$. For this reason, in the rest of the paper we will sometimes treat the content of the each register as an infinite binary sequence. Since not all the binary sequences can be represented in this way, whenever we need to treat a register as an infinite binary sequence, we will do so by representing the sequence $t:\omega\rightarrow 2$ by the real $r$ whose binary representation is such that for every $n\in \mathbb{N}$ the $(2n+1)$st bit is $t(n)$ and all the bits in even position are $0$. In this case we will call $t$ the \emph{binary sequence represented} by $r$.

Given a set $A\subseteq \omega$ we say that it is \emph{$\Gamma$-writable} if there is a $\BSS$-program whose $\Gamma$-computation with no input outputs a real $r$ such that for every $n$, the $n$th bit in the binary sequence represented by $r$ is $1$ if and only if $n\in A$. Similarly, we say that a countable ordinal \emph{$\alpha$ is $\Gamma$-writable} if there is a $\Gamma$-writable set $A$ such that $(\alpha,<)\cong (\mathbb{N},\{(n,m)\;:\; \rho(n,m)\in A\})$. 

Let $P$ be a $n$-register\footnote{Since the number of registers is not of importance in our results, and because of Lemma \ref{Lem:Universal}, in the rest of the paper we will just refer to $n$-registers $\BSS$-programs as $\BSS$-programs.} $\BSS$-program, $r_1,\ldots,r_n\in \mathbb{R}$ be a real numbers, and $(C_\alpha)_{\alpha<\Theta}=(R_1(\alpha)\ldots R_n(\alpha),I(\alpha))_{\alpha< \Theta}$ be the $\SITBM$-computation of $P$ on input $r_1,\ldots,r_n$. For every $\alpha<\Theta$ we will call $C_\alpha$ the \emph{snapshot} of the execution at time $\alpha$. Given an $\BSS$-program $P$ and the $\SITBM$-computation $((R_1(\alpha),\ldots,R_n(\alpha),I(\alpha)))_{\alpha<\Theta}$ of $P$ on input $R_1(0),\ldots,R_n(0)$, for every $i\in \{1,\ldots,n\}$ and $\alpha<\Theta$ we will denote by $R^{\alpha}_i:\alpha \rightarrow \mathbb{R}$ the $\alpha$-sequence such that $R^{\alpha}_i(\beta)=R_i(\beta)$ for all $\beta<\alpha$. In the rest of the paper we will omit the superscript $\alpha$ when it is clear from the context. 

As Lemma \ref{LEM:EqITBM-WITBM}, Lemma \ref{LEM:EqSITBM-BSITBM}, and Lemma \ref{LEM:EqSITBM-SSITBM} show, $\SSITBM$s, $\BSITBM$s, and $\WITBM$s are inessential modifications of $\ITBM$s and $\SITBM$s. For this reason in the rest of the paper we focus on the study of the computational strength of $\SITBM$s. It is worth noticing that a version of $\SITBM$s was briefly considered by Welch in \cite{Welch2014}, see, \S\ref{Sec:Low}.

\begin{lemma}\label{Lem:ITBM-Strong-ITBM}
Every $\ITBM$-computable function is $\SITBM$-computable. 
\end{lemma}
\begin{proof}
The claim follows by Remark \ref{Rem1}.
\end{proof}
\begin{remark}\label{Rem:ITRM}
Since every $\ITRM$ program is essentially a $\BSS$-program, every $\ITRM$-computable function is $\SITBM$-computable. Moreover, the simulation can be performed in exactly the same number of steps.
\end{remark}
By Remark \ref{Rem:ITRM} the result of Lemma \ref{Lem:ITBM-Strong-ITBM} cannot be reversed.
\begin{lemma}\label{Lem:ITBM-Not-Strong-ITBM}
There is a $\SITBM$-computable function which is not $\ITBM$-computable. 
\end{lemma}
\begin{proof}
It is enough to note that $\omega^\omega$ is $\ITRM$-computable \cite[Lemma 1 \& Theorem 7]{Carl2008} and therefore $\SITBM$-writable but not $\ITBM$-writable.
\end{proof}

\begin{lemma}\label{LEM:EqITBM-WITBM}
A real function $f$ is $\ITBM$-computable iff it is $\WITBM$-computable.
\end{lemma} 
\begin{proof}
Trivially every $\ITBM$-computable function is $\WITBM$-computable. Now, let $f$ be any ITBM computable field isomorphism between $\mathbb{R}$ and $(0,1)$, e.g., $G:x\mapsto \frac{e^{2x}}{e^{2x}+1}$. Let $f$ be a $\WITBM$ computable function and $P$ be a program computing it. One can define a program $P'$ in which every constant $c$ is replaced by $G(c)$ and all the fields operations are replaced with the correspondent operations in $(0,1)$. Moreover, at each step of the computation the program checks if one of the registers is $0$ or $1$, if so the program sets the register to $G(0)$. Finally the program should compute $G^{-1}$ of the output. It is not hard to see that the program computes $f$. 
\end{proof}

\begin{lemma}\label{LEM:EqSITBM-BSITBM}
A function $f:\mathbb{R}\rightarrow \mathbb{R}$ is $\SITBM$-computable iff it is $\BSITBM$-computable.
\end{lemma} 
\begin{proof}
The proof is very similar to that of Lemma \ref{LEM:EqITBM-WITBM}. Every $\BSITBM$-computable function is $\SITBM$-computable. Now, let $f$ be any $\ITBM$-computable field isomorphism between $\mathbb{R}$ and $(0,1)$, e.g., $G:x\mapsto \frac{e^{2x}}{e^{2x}+1}$. Let $f$ be an $\SITBM$-computable function and $P$ be a program computing it. One can define a program $P'$ in which every constant $c$ is replaced by $G(c)$ and all the fields operations are replaced with the correspondent operations in $(0,1)$. For each register $R_i$ the machine should have two auxiliary registers $C_i$ and $D_i$. Each time $R_i$ is modified the machine does the following:

If $D_i=0$ and the new value of $R_i$ is smaller than the old value then set $D_i=1$ and $C_i=0$, if $D_i=0$ and the new value of $R_i$ is bigger or equal to the old value then set $C_i=1$, if $D_i=1$ and the new value of $R_i$ is smaller than or equal to the old value then set $C_i=1$, if $D_i=1$ and the new value of $R_i$ is bigger than the old value then set $D_i=0$ and $C_i=0$\footnote{This algorithm checks if the content of $R_i$ is due to a proper limit to infinity.}.

At each step of the computation the machine should check that all the registers of the original program are not $0$ or $1$. If a register $R_i$ is $0$ or $1$ and $C_i\neq 0$ then the machine sets $R_i$ to $G(0)$ and continues the computation otherwise the program enters an infinite loop. 

Finally, the program should compute $G^{-1}$ of the output. It is not hard to see that the program computes $f$. 
\end{proof}

Finally, we note that, once more by shrinking the computation to $(0,1)$ as we did in the proof of Lemma \ref{LEM:EqSITBM-BSITBM}, one can show that $\SSITBM$s can be simulated by $\SITBM$s, and therefore the two models compute exactly the same functions.

\begin{lemma}\label{LEM:EqSITBM-SSITBM}
A function $f:\mathbb{R}\rightarrow \mathbb{R}$ is $\SITBM$-computable iff it is $\SSITBM$-computable.
\end{lemma}

We end this section by noticing that, coding finitely many real numbers in one register one can easily build a universal $\SITBM$ machine. 

\begin{lemma}\label{Lem:Universal}
There is a $\BSS$-program $U$ which given a real number $r$ coding a $\BSS$-program $P$ and the values $\overline{r}$ of the input registers of $P$, executes $P$ on $\overline{r}$. 
\end{lemma}

\section{Bounding the Computational Power of $\SITBM$s}

In this section we will provide a lower bound to the computational strength of $\SITBM$s.

In particular we shall show that $\SITBM$s are stronger than $\ITRM$s. Given $X\subseteq \omega$ we will denote by $\mathcal{O}^X$ the hyperjump of $X$, see, e.g., \cite{sacks2017}.

In order to show that $\SITBM$s are stronger than $\ITRM$s we will prove that $\SITBM$s can compute the $\omega$ iteration of the hyperjump function and much more. First, we note that as in \cite[Proposition 2.8]{Koepke2017}, if a function $f$ is $\SITBM$-computable, then iterations of $f$ along a $\SITBM$-clockable ordinal are also computable. Let $f:\mathbb{R}\rightarrow \mathbb{R}$ be a function, $\Theta$ be an infinite countable ordinal, and $h:\omega  \rightarrow \Theta$ be a  bijection, we define:
\begin{align*}
f^{0}_h&=0,\\
f^{\alpha+1}_h&=f(f^{\alpha}_h) \text{ for $\alpha+1\leq \Theta$},\\
f^{\lambda}_h&=\bigoplus^{h}_{\alpha\in \lambda}f^{\alpha}_h \text{ for $\lambda\leq\Theta$ limit},
\end{align*}
where $\bigoplus^{h}_{\alpha\in \lambda}f^{\alpha}_h$ is the real $r$ such that if $n=\rho(i,h(\alpha))$ then the $n$th bit of the binary sequence represented by $r$ is the same as the $i$th bit of the binary representation of $f^{\alpha}_h$ for all $i\in\mathbb{N}$ and $\alpha<\lambda$, and $0$ for $\lambda<\alpha$. 

Similarly to \cite[Proposition 2.7]{Koepke12}, one can easily see that $\SITBM$s can compute iterations of $\SITBM$-computable function over an $\SITBM$-writable ordinal. 

\begin{lemma}[Iteration Lemma]\label{Lem:Iter}
Let $f:\mathbb{R}\rightarrow \mathbb{R}$ be a $\SITBM$-computable function, and $\Theta$ be $\SITBM$-writable. Then there is a bijection $h:\omega\rightarrow\Theta$ such that $f^{\Theta}_h$ is $\SITBM$-computable.
\end{lemma}
\begin{proof}[Sketch]
We will only sketch the proof and we leave the details to the reader. 

Assume that $\Theta$ is $\SITBM$-writable. Let $h$ be such that $h:(\mathbb{N},\{(n,m)\;:\; \text{the }\rho(n,m) \text{th bit of the binary sequence represented by } r \text{ is } 1\}) \cong (\Theta,<)$ where $r$ is the $\SITBM$-writable real coding $\Theta$. There is a machine that computes $f^{\Theta}_h$.

Note that, since $\Theta$ is $\SITBM$-writable, essentially by using the classical $\ITTM$ algorithm, given a set of natural numbers we can always compute their least upper bound (if it exists) according to the order induced by $h$. Therefore, we can generate the sequence $h^{-1}(0)h^{-1}(2)\ldots h^{-1}(\omega)\ldots$. 

As noted in \cite[Proposition 2.8]{Koepke2017} the main problem in proving this lemma is to show that we can iterate the same program infinitely many times ensuring that no register in the program will diverge because of the iteration\footnote{Note that, while the example mentioned in \cite[p.42]{Koepke2017} is not problematic for $\SITBM$s, our machines could still diverge if for example in the infinite iteration one of the registers of the program assumes values $0,1,-1,2,-2,3,-3,\ldots$.}. To solve this problem we note that by Lemma \ref{LEM:EqSITBM-SSITBM} it is enough to show that the function is $\SSITBM$-computable. Note that $\SSITBM$-computations are obviously closed under transfinite iterations since $\SSITBM$s never brake because of a diverging register.

 
The $\SSITBM$ algorithm to compute the function is the following: the machine first writes $\Theta$ in one of the registers say $R_1$. Then, our machine starts computing $f(0)$ and saves the result in the register $R_2$ by copying the $i$th bit of $f(0)$ in position $\rho(i,h^{-1}(0))$ of the binary sequence represented by $R_2$. In general, the machine should compute $f(f^{\alpha}_h)$ and save the result in $R_2$ according to $h^{-1}$ as we did for $f(0)$. Checking the correctness of the  sketched algorithm is left to the reader. 
\end{proof}

Now, we can prove that $\SITBM$s are strictly stronger than $\ITRM$s. First note that since $\ITRM$s can compute the hyperjump of a set, by Remark \ref{Rem:ITRM}, so can $\SITBM$s.

\begin{lemma}
Let $X\subseteq \omega$. The hyperjump $\mathcal{O}^X$ in the oracle $X$ is $\SITBM$-computable. 
\end{lemma}

By Lemma \ref{Lem:Iter} we have the following result:
 
\begin{lemma}
Let $\alpha$ be $\SITBM$-clockable. Then the $\alpha$th iteration of the hyperjump is $\SITBM$-computable. 
\end{lemma}

But since by \cite[Proposition 12]{Koepke2009OC} $\ITRM$s cannot compute the $\omega$-iteration of the hyperjump we have that $\SITBM$s are strictly stronger than $\ITRM$s.

\begin{corollary}
There is a subset of natural numbers $A\subseteq \mathbb{N}$ which is $\SITBM$-computable but not $\ITRM$-computable, and thus $A\notin \mathbf{L}_{\omega^{\mathrm{CK}}}$.
\end{corollary}

We end this section by proving a looping criterion for the divergence of $\SITBM$s which is analogous to those for $\ITTM$s \cite[Theorem 1.1]{Hamkins}, $\ITRM$s \cite[Theorem 5]{Carl2008}, and %
$\ITBM$s \cite[Theorem 1]{Koepke12}.

\begin{lemma}[Strong Loop Lemma]\label{Lem:strongLoop}
Let $P$ be a $\BSS$-program, $(C_\alpha)_{\alpha< \Theta}=((R_1(\alpha),\ldots,R_n(\alpha), I(\alpha)))_{\alpha<\Theta}$ be a $\SITBM$-computation of $P$, and let $\gamma<\beta<\Theta$ be such that:
\begin{enumerate}
\item \label{SL1} $(R_1(\gamma),\ldots,R_n(\gamma),I(\gamma))=(R_1(\beta),\ldots,R_n(\beta), I(\beta))$;
\item \label{SL2}  for all $\gamma\leq \alpha\leq \beta$ we have $I(\beta)\leq I(\alpha)$ and for all $i\in \{1\ldots n\}$ we have $R_i(\beta)\leq R_i(\alpha)$.
\end{enumerate}
Then $P$ diverges. 
\end{lemma}
\begin{proof}
Without loss of generality we will assume that $n=1$, a similar proof works for an arbitrary number of registers. Note that since $(R_1(\gamma),I(\gamma))=(R_1(\beta),I(\beta))$ the machine is in a loop. Let us call $L=((R_1(\alpha),I(\alpha)))_{\gamma\leq \alpha\leq \beta}$ the \emph{looping block of the computation}. Let $\delta$ be the smallest ordinal such that $\gamma+\delta=\beta$, we will call $\delta$ the \emph{length} of $L$. 

\begin{claimN}\label{Calim1}
If $\alpha=\gamma+\delta\times \nu$ for some $\nu>0$ is such that $(R_1(\gamma), I(\gamma))=(R_1(\alpha),I(\alpha))$ then for all $\mu<\delta$ we have $(R_1(\gamma+\mu),I(\gamma+\mu))=(R_1(\alpha+\mu),I(\alpha+\mu))$, i.e., if $\alpha=\gamma+\delta\times \nu$ is such that $(R_1(\gamma),I(\gamma))=(R_1(\alpha),I(\alpha))$ then the computation from $\alpha$ to $\alpha+\delta$ is the loop $L$. 
\end{claimN}
\begin{proof}
We prove the claim by induction on $\mu$. If $\mu=0$ the claim follows trivially from the hypothesis. If $\mu=\eta+1$ then by inductive hypothesis $(R_1(\gamma+\eta),I(\gamma+\eta))=(R_1(\alpha+\eta),I(\alpha+\eta))$. But then $(R_1(\gamma+\eta+1),I(\gamma+\eta+1))=(R_1(\alpha+\eta+1),I(\alpha+\eta+1))$ follows from the fact that our machines are deterministic. Finally for $\mu$ limit the claim follows from the inductive hypothesis and from Remark \ref{Rem2}. 
\end{proof}
\begin{claimN}\label{Claim2}
For all $\nu>0$ we have $(R_1(\gamma),I(\gamma))=(R_1(\gamma+\delta\times \nu),I(\gamma+\delta\times \nu))$.
\end{claimN}
\begin{proof}
We prove the claim by induction on $\nu$. If $\nu=1$ the claim follows from the assumptions. For $\nu=\eta+1$ then the claim follows from Claim \ref{Calim1}. Assume that $\nu$ is a limit ordinal. By inductive hypothesis and by Claim \ref{Calim1} the computation from step $\gamma$ to step $\gamma+\delta\times \nu$ consists of $\nu$-many repetitions of the loop $L$. In particular this means that the snapshot $(R_1(\gamma),I(\gamma))$ appears cofinally often in the computation up to $\nu$. Therefore $R_1(\gamma)\in \SLimP(R^{\gamma+\delta\times \nu}_1)$. Finally, by \ref{SL2} we have that $I(\gamma)=\liminf_{\alpha<\gamma+\delta\times \nu}I(\gamma)=I(\gamma+\delta\times \nu)$ and $R_1(\gamma)=\min(\SLimP(R^{\gamma+\delta\times \nu}_1))=R_1(\gamma+\delta\times \nu)$ as desired. 
\end{proof}
Finally, note that Claim \ref{Claim2} implies that the computation of $P$ diverges as desired.
\end{proof}

We call a computation $L=((R_1(\alpha),\ldots,R_n(\alpha),I(\alpha)))_{\alpha<\Theta}$ as the one in Lemma \ref{Lem:strongLoop} a \emph{strong loop}.


\section{Low Complexity Machines}\label{Sec:Low}

In the rest of the paper we investigate the ordinals which are clockable by $\SITBM$s whose computations are of low complexity. In particular, we strengthen Lemma \ref{Lem:Welch} which was mentioned without proof in \cite{Welch2014}\footnote{In \cite[p.31]{Welch2014} Welch mentions that the first $\Pi_3$-reflecting ordinal is an upper bound to the computational strength of $\SITBM$s. In a private communication he clarified to the authors that the machines he was referring to are actually machines which only work on rational numbers. The question of whether the first $\Pi_3$-reflecting ordinal is an upper bound to the computational strength of $\SITBM$s is still open.}.


\begin{theorem}[Welch]\label{Lem:Welch}
Let $\beta$ be the first $\Pi_3$-reflecting ordinal. Then, for every rational $\BSS$-program $P$ we have that the $\SITBM$-computation of $P$ with input $0$ either diverges or halts before $\beta$.
\end{theorem}

\begin{lemma}\label{Lem:PI3CofOften}
Let $\beta$ be $\Pi_3$-reflecting, $\Theta>\beta$, $r\in \mathbb{R}\cap \mathbf{L}_\beta$ be a real number, and $(C_\alpha)_{\alpha\in \Theta}$ be an $\SITBM$-computation with $R_i(\beta)=r$. Then $R_i$ has value $r$ cofinally often below $\beta$, i.e., for all $\gamma<\beta$ there is $\gamma<\alpha<\beta$ such that $R_i(\alpha)=r$. In fact, there are cofinally in $\beta$ many $\tau$ such that $C_{\tau}=C_{\beta}$.
\end{lemma}
\begin{proof}
Note that for every $\tau<\beta$ the following sentences $\phi_{i}:=\forall r'<r \exists \alpha>\tau \forall \gamma>\alpha R_i(\gamma)>r'$, and $\psi_{i}:=\forall r'>r \forall \alpha>\tau \exists \gamma>\alpha R_i(\gamma)<r'$  are $\Pi_3$ in $\mathbf{L}_\beta$. 
Moreover, they are both true in $\mathbf{L}_\beta$ since $R_i(\beta)=r$. Consequently, there are cofinally in $\beta$ many $\tau$ such that both statements hold in $\mathbf{L}_{\tau}$. However, this implies that $R_{i}(\tau)=R_{i}(\beta)$. Thus, for every $\tau<\beta$ there is $\tau<\alpha<\beta$ such that $R_i(\alpha)=r$. 

Now let $n$ be the maximal register index appearing in $P$. Then the conjunction $\bigwedge_{i\leq n}(\phi_{i}\wedge\psi_{i})$ is still $\Pi_3$ and thus holds at cofinally often in $\beta$. Thus, the register contents at time $\beta$ have appeared cofinally often before time $\beta$. Modifying $P$ slightly to $P^{\prime}$ by storing the active program line in an additional register, we see that the same actually holds for the whole configuration.
\end{proof}

\begin{theorem}\label{Thm:Pi3Bound}
Let $\beta$ be a $\Pi_3$-reflecting ordinal. Then, for every $(\mathbf{L}_\beta \cap \Rea)$-$\SITBM$ computation $(C_\alpha)_{\alpha\in \Theta}$ we have that either $\Theta=\Ord$ or $\Theta<\beta$. Moreover, if $\Theta=\Ord$ then the machine is in a strong loop.
\end{theorem}
\begin{proof}
Assume that $\Theta>\beta$. By Lemma \ref{Lem:PI3CofOften}, we have that the snapshot $C_\beta$ appears cofinally often before $\beta$. Now, we want to show that this means that the program is in a strong loop. Let $\delta<\beta$ be such that $C_\delta=C_\beta$. If there is no strong loop between times $\delta$ and $\beta$, there are a register index $i$ and an ordinal $\gamma$ such that $\delta+\gamma<\beta$ and $r:=R_i(\delta+\gamma)<R_i(\delta)$. Note, however, that the snapshot of a computation at time $\eta+\gamma$ is determined by the snapshot at time $\eta$. Thus, for every $\tau<\beta$ such that $C_{\tau}=C_{\beta}$, we have $R_{i}(\tau+\gamma)=R_{i}(\delta+\gamma)=r$. Since $\beta$ is $\Pi_3$-reflecting, we have $\tau+\gamma<\beta$ for every $\tau<\beta$. Thus, the content of $R_i$ is equal to $r$ cofinally often before $\beta$. Consequently, we have $r<R_{i}(\beta)=\text{liminf}_{\iota<\beta}R_{i}(\iota)\leq r$, a contradiction. Thus, the computation is in a strong loop. The claim follows by Lemma \ref{Lem:strongLoop}.
\end{proof}

In \cite[p.31]{Welch2014} Welch asked if the bound of Lemma \ref{Lem:Welch} was optimal. The following lemma shows that this is not the case.

\begin{lemma}
The supremum of the $\SITBM$-clockable ordinals is not a $\Pi_3$-reflecting ordinal.
\end{lemma} 
\begin{proof}
It is enough to note that the sentence expressing the fact that every $\SITBM$-computation either diverges or stops is $\Pi_3$, and can therefore be reflected below any $\Pi_3$-reflecting ordinal.
\end{proof}

\begin{corollary}
Let $\alpha$ be an ordinal. The supremum of the ordinals clockable by an $(\mathbf{L}_\alpha\cap \Rea)$-$\SITBM$  is strictly smaller than the first $\Pi_3$-reflecting ordinal above $\alpha$. 
\end{corollary}

A $\BSS$-program \emph{slow} if its $\SITBM$-computation $(C_\beta)_{\beta\in \Theta}$ is such that for every $\omega\leq \beta\leq \Theta$ we have $C_\beta\in \mathbf{L}_\beta$. 

\begin{corollary}
For every $\alpha$ the supremum of the ordinals that are $(\mathbf{L}_\alpha\cap \Rea)$-$\SITBM$-clockable by a slow program is smaller than or equal to the first $\Pi_3$-reflecting ordinal.
\end{corollary}

Let $P$ be a $\BSS$-program whose $\SITBM$-computation $(C_\beta)_{\beta\in \Ord}$ diverges. We will call the \emph{strong looping time} of $P$ the least ordinal $\alpha$ such that there is $\beta<\alpha$ and $(C_\gamma)_{\beta\leq \gamma\leq \alpha}$ is a strong loop. Analogously to what was proved in \cite[Theorem 5]{Carl2008} for $\ITRM$, one can prove that strong loops characterise divergent $\SITBM$-computation.

\begin{lemma}
Every program $P$ whose $\SITBM$-computation $(C_\beta)_{\beta\in \Ord}$ diverges has a strong looping time. 
\end{lemma}
\begin{proof}
Take $\beta=\omega_1$ in Theorem \ref{Thm:Pi3Bound}\footnote{It is a classical result of set theory that $\omega_1$ is a $\Pi_3$-reflecting ordinal. Indeed, via a classical hull construction and an application of the Condensation Lemma one can easily find an $\alpha\in \omega_1$ such that $\mathbf{L}_\alpha\prec \mathbf{L}_{\omega_1}$.}. Then, every divergent $(\mathbb{L}_{\omega_1}\cap \mathbb{R}^{\mathbf{L}})$-$\SITBM$-computation must have a strong loop. Finally, it is enough to note that since $\mathbb{L}_{\omega_1}\cap \mathbb{R}=\mathbb{R}^{\mathbf{L}}$ we have that $(\mathbb{L}_{\omega_1}\cap \mathbb{R}^{\mathbf{L}})$-$\SITBM$-computation and $\SITBM$-computations coincide.
\end{proof}

The following is an analogue of the observation in \cite{Hamkins} that the supremum $\Sigma$ of the looping times for $\ITTM$s is not admissible.

\begin{lemma}
The supremum $\gamma$ of the strong looping times is not $\Pi_2$-reflecting thus in particular not admissible.
\end{lemma} 
\begin{proof}
Suppose otherwise. Let $U$ be the universal program introduced above. Since $U$ does not halt, it will loop. By definition of $\gamma$, the first loop of $U$ will be finished exactly at time $\gamma$.  Thus, there is $\iota<\gamma$ such that the configurations of $U$ at times $\iota$ and $\gamma$ agree and are component-wise (weakly) majorized by all configurations occuring in between. Let $\phi$ denote the statement ``for all $\xi>\iota$, the configuration at time $\xi$ is component-wise greater than or equal to that at time $\iota$"; clearly, $\phi$ is $\Pi_2$.

Now let $i$ be the index of a register used in $U$, and let $r$ be its content at time $\iota$ (and thus at time $\gamma$).
Now let $\psi_{i}$ be the statement ``for all rational $q>r$ and all $\tau$, there is $\zeta>\tau$ such that, at time $\zeta$, the content of $R_{i}$ is $<q$", which is again $\Pi_2$. 

Now, the conjunction $\tilde{\phi}$ of $\phi$ and all the $\psi_{i}$ is a $\Pi_2$-formula expressing that $U$ strongly loops. Since $\gamma$ is $\Pi_2$-reflecting by assumption, there is $\tau<\gamma$ such that $L_{\tau}\models\tilde{\phi}$. But then, $U$ already loops at time $\tau$, contradicting the definition of $\gamma$.

\end{proof}

We end this section by showing that the model of computation obtained by restricting $\SITBM$s to reals in  $\BSS$-programs is strictly weaker than the unrestricted one.

\begin{lemma}\label{Lem:CompSnapLemma}
Let $\beta$ be such that there is an $\SITBM$-computable code for $\mathbf{L}_{\beta}$. 
There is a $\BSS$-program $P$ whose $\SITBM$-computation decides the halting problem for $(\mathbf{L}_{\beta}\cap\mathbb{R})$-$\SITBM$s, i.e., given the code $c$ for a $\BSS$-program, $P$ halts with output $0$ if the $(\mathbf{L}_{\beta}\cap\mathbb{R})$-$\SITBM$-computation of the program coded by $c$ with input $0$ halts, and $P$ halts with output $1$ if the $(\mathbf{L}_{\beta}\cap\mathbb{R})$-$\SITBM$-computation of the program coded by $c$ with input $0$ diverges. (If $c$ codes a program that does not produce an $(\mathbf{L}_{\beta}\cap\mathbb{R})$-$\SITBM$-computation, we make no statement about the behaviour of $P$, although it is easily possible to arrange $P$ to, e.g.,  diverge or halt with output $2$ in this case.)
\end{lemma}
\begin{proof}
Let $U$ be the universal $\SITBM$ machine of Lemma \ref{Lem:Universal}. We will describe the program $P$ and leave the precise implementation to the reader.  We start by computing a real number $b$ coding $\mathbf{L}_{\beta}$.

Given a code $c$ for a $\BSS$-program $P'$ as input $P$ starts executing $U$ on $c$. The program $P$ will use the register $R$ to keep track of the snapshots that appeared in the computation. At each step of the computation $P$ first checks if the simulation halted, in which case halts with output $0$. If the current snapshot $C$ of the simulation of $P'$ is not an halting snapshot, $P$ will use $b$ to find a natural number $n$ that codes $C$ in the sense of $b$.

Once $P$ finds the code for the snapshot it will check if the $n$th bit of binary sequence represented by the real in register $R$ is one, in which case the $P$ halts with output $1$. If not then for every $i\in \mathbb{N}$ the program $P$ checks if the $i$th bit of the binary sequence represented by the register register $R$ is $1$ in this case computes the snapshot coded by $i$ and checks that every element of the snapshot is smaller than or equal to  the corresponding element in the snapshot coded by $n$. If this is the case 
then $P$ sets the $(2n+1)$st bit of $R$ to $1$ and continues. If not then machine erases $R$, set the $(2n+1)$st bit of $R$ to $1$ and continues.

One can check that the $P$ does compute the halting problem restricted to $(\mathbf{L}_{\beta}\cap\mathbb{R})$-$\ITBM$s.
\end{proof}

As usual a subset $A$ of real numbers is called \emph{$\SITBM$-semi decidable} if there is a $\BSS$-program $P$ such that for every $r\in A$ the $\SITBM$-computation of $P$ with input $r$ halts, and for every $r\notin A$ the $\SITBM$-computation of $P$ with input $r$ diverges. Moreover, if $P$ can be chosen in such a way that for every $r\in A$ the $\mathbb{Q}$-$\SITBM$-computation of $P$ with input $r$ halts, and for every $r\notin A$ the $\mathbb{Q}$-$\SITBM$-computation of $P$ with input $r$ diverges. Then we will say that $A$ is \emph{rationally $\SITBM$-semi decidable}.

\begin{corollary}
Every rationally $\SITBM$-semi decidable set is $\SITBM$-computable. 
\end{corollary}

\begin{corollary}
There is a $\SITBM$-semi decidable set which is not rationally $\SITBM$-semi decidable. 
\end{corollary}

\bibliographystyle{splncs}
\bibliography{REFERENCES}

\begin{thebibliography}{10}
\providecommand{\url}[1]{\texttt{#1}}
\providecommand{\urlprefix}{URL }
\providecommand{\doi}[1]{https://doi.org/#1}

\bibitem{BSS}
Blum, L., Shub, M., Smale, S.: On a theory of computation and complexity over
  the real numbers: $\mathrm{NP}$-completeness, recursive functions and
  universal machines. Bulletin of the American Mathematical Society
  \textbf{21},  1--46 (1989)

\bibitem{CarlBook}
Carl, M.: Ordinal Computability. An Introduction to Infinitary Machines. De
  Gruyter (2019)

\bibitem{CarlComp}
Carl, M.: Space and time complexity for {I}nfinite {T}ime {T}uring {M}achines
  (2019), arXiv:1905.06832

\bibitem{CarlZOOII}
Carl, M.: Taming {K}oepke's zoo {II}: Register machines (2019),
  arXiv:1907.09513

\bibitem{Carl2008}
Carl, M., Fischbach, T., Koepke, P., Miller, R., Nasfi, M., Weckbecker, G.: The
  basic theory of infinite time register machines. Archive for Mathematical
  Logic  \textbf{49}(2),  249 -- 273 (2010)

\bibitem{GaleottiPhD}
Galeotti, L.: The theory of generalised real numbers and other topics in logic.
  Ph.D. thesis, Universit{\"a}t Hamburg (2019)

\bibitem{Galeotti2019}
Galeotti, L.: Surreal blum-shub-smale machines. In: Manea, F., Martin, B.,
  Paulusma, D., Primiero, G. (eds.) Computing with Foresight and Industry. pp.
  13--24. Springer International Publishing, Cham (2019)

\bibitem{Hamkins}
Hamkins, J.D., Lewis, A.: Infinite time {Turing} machines. Journal of Symbolic
  Logic  \textbf{65},  567--604 (2000). \doi{10.2307/2586556}

\bibitem{Koepke2017}
Koepke, P., Morozov, A.S.: The computational power of infinite time
  {B}lum-{S}hub-{S}male machines. Algebra and Logic  \textbf{56}(1),  37--62
  (2017)

\bibitem{Koepke09}
Koepke, P., Seyfferth, B.: Ordinal machines and admissible recursion theory.
  Annals of Pure and Applied Logic  \textbf{160}(3),  310 -- 318 (2009)

\bibitem{Koepke2009OC}
Koepke, P.: Ordinal computability. In: Ambos-Spies, K., L{\"o}we, B., Merkle,
  W. (eds.) Mathematical Theory and Computational Practice. pp. 280--289.
  Springer Berlin Heidelberg, Berlin, Heidelberg (2009)

\bibitem{Koepke12}
Koepke, P., Seyfferth, B.: Towards a theory of infinite time
  {B}lum-{S}hub-{S}male {M}achines. In: Cooper, S.B., Dawar, A., L{\"o}we, B.
  (eds.) How the World Computes: Turing Centenary Conference and 8th Conference
  on Computability in Europe, CiE 2012, Cambridge, UK, June 18-23, 2012.
  Proceedings. vol.~7318, pp. 405--415. Springer (2012)

\bibitem{sacks2017}
Sacks, G.E.: Higher Recursion Theory. Perspectives in Logic, Cambridge
  University Press (2017). \doi{10.1017/9781316717301}

\bibitem{Seyfferth}
Seyfferth, B.: Three Models of Ordinal Computability. Ph.D. thesis, Rheinische
  Friedrich-Wilhelms-Universit{\"a}t Bonn (2013)

\bibitem{Welch2014}
Welch, P.D.: Turing's legacy: developments from {T}uring's ideas in logic.
  Lecture Notes in Logic (42),  493–--529 (2014)

\bibitem{Welch2016}
Welch, P.D.: Discrete transfinite computation. In: Sommaruga, G., Strahm, T.
  (eds.) Turing's Revolution: the impact of his ideas about computability. pp.
  161--185. Springer Verlag, Basel (2016)

\end{thebibliography}
\newpage

\end{document}